\documentclass[oneside,english]{amsart}
\pdfoutput=1
\usepackage[T1]{fontenc}
\usepackage[latin9]{inputenc}
\usepackage{amsthm}
\usepackage{amstext}
\usepackage{graphicx}

\makeatletter
\numberwithin{equation}{section}
\numberwithin{figure}{section}
\theoremstyle{plain}
\newtheorem{thm}{Theorem}[section]
  \theoremstyle{remark}
  \newtheorem{rem}[thm]{Remark}
  \theoremstyle{plain}
  \newtheorem{prop}[thm]{Proposition}
  \theoremstyle{plain}
  \newtheorem{lem}[thm]{Lemma}
  \theoremstyle{plain}
  \newtheorem{cor}[thm]{Corollary}
 \theoremstyle{definition}
  \newtheorem{example}[thm]{Example}

\makeatother

\usepackage{babel}

\begin{document}

\title{Intersections of Certain Deleted Digits Sets}

\author{Steen Pedersen and Jason Phillips}

\address{Department of Mathematics, 3640 Colonel Glenn Highway, Wright State University, Dayton OH 45435.}

\email{steen@math.wright.edu}

\email{phillips.50@wright.edu}

\begin{abstract}
We consider some properties of the intersection of deleted digits
Cantor sets with their translates. We investigate conditions on the
set of digits such that, for any \emph{t} between zero and the dimension
of the deleted digits Cantor set itself, the set of translations such
that the intersection has that Hausdorff dimension equal to \emph{t}
is dense in the set \emph{F} of translations such that the intersection
is non-empty. We make some simple observations regarding properties
of the set \emph{F}, in particular, we characterize when \emph{F}
is an interval, in terms of conditions on the digit set. 
\end{abstract}
\maketitle

\section{Introduction}

Let $n\geq3$ be an integer. Any real number $0\leq x\leq1$ can be
written in base $n$ as an $n-$ary expansion \begin{equation}
x=0.x_{1}x_{2}\cdots:=\sum_{k=1}^{\infty}\frac{x_{k}}{n^{k}}\label{eq:n-ary}\end{equation}
where $x_{k}\in\{0,1,\ldots,n-1\}.$ This representation of a real
number $x$ in the interval $[0,1]$ is unique, except\[
0.x_{1}x_{2}\cdots x_{k}=0.x_{1}x_{2}\cdots x_{k-1}y_{k}y_{k+1}\cdots\]
when $x_{k}\neq0,$ $y_{k}=x_{k}-1,$ and $y_{j}=n-1$ for $j>k.$ 

Let $\mathcal{D}=\{d_{1},d_{2},\ldots,d_{m}\}$ be a set of at least
two integers, such that $0=d_{1}<d_{2}<\cdots<d_{m}<n$ and $m<n.$
The set \[
\mathcal{C}=\mathcal{C}_{n,\mathcal{D}}:=\left\{ \sum_{k=1}^{\infty}\frac{x_{k}}{n^{k}}\,\Big|\, x_{k}\in\mathcal{D}\right\} \]
is a \emph{deleted digits Cantor set}. Consequently, $\mathcal{C}_{n,\mathcal{D}}$
is obtained from the set of all $n-$ary representations (\ref{eq:n-ary})
by restricting attention to those $n-$ary representations that only
contain digits from the set $\mathcal{D},$ that is by deleting the
digits not in $\mathcal{D}$ from the set of all potential digits
$\left\{ 0,1,\ldots,n-1\right\} .$  

An \emph{$n-$ary interval} is an interval of the form \[
\left[\frac{k}{n^{j}},\frac{k+1}{n^{j}}\right],\]
where $j\geq0$ is an integer and $k=0,1,\ldots,n^{j}-1.$ Due to
the structure of deleted digits Cantor sets, it is natural to work
with the Minkowski dimension. Let $S$ be a subset of the closed interval
$[0,1].$ Then the \emph{Minkowski dimension} of $S$ is \[
\dim_{\mathrm{M}}S=\lim_{k\to\infty}\frac{\log\mathcal{N}_{j}(S)}{j\log n}\]
where $\mathcal{N}_{j}(S)$ is the minimum number of $n-$ary intervals
of length equal to $1/n^{j}$ needed to cover $S.$ If the limit does
not exists, we can talk about the upper and lower Minkowski dimensions,
obtained by replacing the limit by the limit superior and the limit
inferior, respectively. Minkowski dimension is sometimes called\emph{
Minkowski\textendash{}Bouligand dimension, box dimension, Kolmogorov
dimension, entropy dimension, }or\emph{ limiting capacity}. 

For a subset $S$ of the interval $[0,1],$ let\[
\Lambda_{\beta}(S):=\lim_{\delta\to0}\inf\left\{ \sum_{I\in\mathcal{U}}|I|^{\beta}\right\} \]
where the infimum is over all coverings $\mathcal{U}$ of $S$ by
$n-$ary intervals whose lengths $|I|$ are at most $\delta>0.$ The
\emph{Hausdorff dimension} of $S$ is the number $\dim_{\mathrm{H}}S$
satisfying \begin{align*}
\Lambda_{\beta}(S) & =0\text{ for \ensuremath{\beta>\dim_{\mathrm{H}}S,\text{ and}}}\\
\Lambda_{\beta}(S) & =\infty\text{ for \ensuremath{\beta<\dim_{\mathrm{H}}S.}}\end{align*}
The term Hausdorff dimension is sometimes replaced by \emph{Hausdorff-Besicovitch
dimension, Besicovitch dimension, or fractional dimension.} It is
well-known, and not difficult to see, that this definition of the
Hausdorff dimension agrees with the standard definition, where the
infimum is over all countable covering of $S$ by intervals of lengths
at most $\delta.$ For example, this was established by Besicovitch
\cite{Bes52}, when $n=2.$ See also the book \cite{Fal85} by Falconer. By the definition
of Minkowski dimension, the Hausdorff dimension of $S$ is bounded
above by the (lower) Minkowski dimension of $S.$ Hutchinson  showed \cite{Hut81} both the Minkowski and Hausdorff dimensions of $\mathcal{C}$ equals
the similarity dimension $\log_{n}m.$

Let \[
\mathcal{F}:=\left\{ t\geq0\mid\mathcal{C}\cap\left(\mathcal{C}+t\right)\neq\emptyset\right\} .\]
be that set of positive real numbers $t,$ such that the intersection
of $\mathcal{C}$ and its translate by $t,$ $\mathcal{C}+t:=\{x+t\mid x\in\mathcal{C}\},$
is non-empty. Let $\mathcal{F}_{\alpha}$ be the set of $t$ in $\mathcal{F}$
such that $\mathcal{C}\cap(\mathcal{C}+t)$ has Hausdorff dimension
$\alpha\log_{n}m.$ Under suitable assumptions on $\mathcal{D}$ we
prove that $\mathcal{F}_{\alpha}$ is dense in $\mathcal{F}$ for
any $0\leq\alpha\leq1.$

Our main results are
\begin{thm}
\label{thm:non-uniform} Let $n$ be a positive integer and let $\mathcal{D}=\{d_{1},d_{2},\ldots,d_{m}\}$
be a set of at least two integers, such that $0=d_{1}<d_{2}<\cdots<d_{m}<n-1$
and $2\leq d_{k+1}-d_{k},$ for $k=1,2,\ldots,m-1.$ Let $\mathcal{C}$
be the real numbers of the form $\sum_{k=1}^{\infty}x_{k}/n^{k}$
where each $x_{k}$ is in $\mathcal{D}.$ For each $0\leq\alpha\leq1,$
the set $\mathcal{M}_{\alpha}$ of all $0\leq t\leq1$ such that the
Hausdorff and Minkowski dimensions of $\mathcal{C}\cap(\mathcal{C}+t)$
equals $\alpha\log_{n}m$ is dense in the set $\mathcal{F}$ of all
$0\leq s\leq1$ such that $\mathcal{C}\cap(\mathcal{C}+s)$ is not
the empty set. 
\end{thm}
When $\mathcal{C}$ is the triadic Cantor set, that is $n=3$ and
$\mathcal{D}=\{0,2\},$ the conclusions of Theorem \ref{thm:non-uniform}
were established by Davis and Hu \cite{DaHu95} and by a different method by Nekka and Li 
\cite{NeLi02}. Of course, this does not follow from Theorem \ref{thm:non-uniform}.
To recover the conclusions of Theorem \ref{thm:non-uniform} for the
triadic Cantor set we consider a class of {}``uniform'' deleted
digits Cantor sets. 

A deleted digits Cantor sets is \emph{uniform} \cite{DaTi08}, \cite{Li-Nekka-04},
if there is an integer $d\geq2,$ such that $d_{j}=d(j-1),$ for $j=1,2,\ldots,m.$
If $d_{m}=n-1,$ then we cannot apply Theorem \ref{thm:non-uniform}.
However, the additional structure on the digit set $\mathcal{D},$
imposed by assuming uniformity, allows us to establish the conclusions
of Theorem \ref{thm:non-uniform} in this case. In fact, these conclusions
hold under slightly weaker assumptions. 
\begin{thm}
\label{thm:uniform} Let $n$ be a positive integer and let $\mathcal{D}=\{d_{1},d_{2},\ldots,d_{m}\}$
be a set of at least two integers, such that $0=d_{1}<d_{2}<\cdots<d_{m}<n.$
Suppose there is an integer $h>1,$ such that each digit $d_{j}$
is an integral multiple of $h.$ Let $\mathcal{C}$ be the real numbers
of the form $\sum_{k=1}^{\infty}x_{k}/n^{k}$ where each $x_{k}$
is in $\mathcal{D}.$ For each $0\leq\alpha\leq1,$ the set $\mathcal{M}_{\alpha}$
of all $0\leq t\leq1$ such that the Hausdorff and Minkowski dimensions
of $\mathcal{C}\cap(\mathcal{C}+t)$ equals $\alpha\log_{n}m$ is
dense in the set $\mathcal{F}$ of all $0\leq s\leq1$ such that $\mathcal{C}\cap(\mathcal{C}+s)$
is not the empty set. 
\end{thm}

In Section \ref{sec:Preliminaries} we show how the intersections
$\mathcal{C}\cap(\mathcal{C}+t)$ can be understood in term of a deleted
intervals construction. This is used in Section \ref{sec:A-Stepping-Stone}
to show that if $\mathcal{D}$ satisfies the separation condition
$d_{j+1}-d_{j}\geq2,$ then for any $0\leq\alpha\leq1,$ there is
a $t$ such that $\mathcal{C}\cap(\mathcal{C}+t)$ has dimension $\alpha\log_{n}m.$
The results in Sections \ref{sec:Preliminaries} and \ref{sec:A-Stepping-Stone}
are then used in Sections \ref{sec:Proof-of-Theorem Main} and \ref{sec:Proof-of-Theorem-uniform}
to establish Theorems \ref{thm:non-uniform} and \ref{thm:uniform}
respectively. In Section \ref{sec:Geometry-of-F} we investigate the
geometry of $\mathcal{F},$ in particular, we obtain a characterization
of when $\mathcal{F}$ is an interval in terms of a property of the
digit set $\mathcal{D}.$ And as an application of our results, we
construct everywhere discontinuous functions mapping the interval
$[0,1]$ into itself. In Section \ref{sec:Concluding-Remarks} we
investigate the necessity of the conditions imposed on the digit set
$\mathcal{D}$ in Theorem \ref{thm:non-uniform}. In Section \ref{sec:Open-Questions}
we state some questions related to the results obtained in this paper. 

For background information on fractal sets and their dimensions we
refer the reader to the book \cite{Fal85} by Falconer. Motivations, including potential
applications in physics, for studying the problems considered in this
paper can, for example, be found in papers by Davis and Hu \cite{DaHu95}, 
Li and Nekka \cite{Li-Nekka-04}, and Dai and Tian \cite{DaTi08}.

\section{Preliminaries\label{sec:Preliminaries}}

The purpose of this section is to show how the deleted intervals construction
of $\mathcal{C}$ can be used to analyse $\mathcal{C}\cap(\mathcal{C}+t).$
These observations form the basis for our proof of Theorem \ref{thm:non-uniform}.

\subsection{\label{sub:Constructions-of-deleted}Constructions of deleted digits
Cantor sets}

Let \[\mathcal{D}=\{d_{1},d_{2},\ldots,d_{m}\}\]
be a set of at least two distinct integers, the set of \emph{digits},
such that $0=d_{1}<d_{2}<\cdots<d_{m}<n.$ The corresponding \emph{deleted
digits Cantor set} $\mathcal{C}$ is the set of $n-$ary real numbers
in $[0,1]$ that can be constructed using only digits from the digit
set $\mathcal{D},$ that is \[
\mathcal{C}=\left\{ 0.x_{1}x_{2}\cdots\,\Big|\, x_{j}\in\mathcal{D}\right\} .\]

\subsubsection{Self-similarity construction of $\mathcal{C}$ }

Let $S_{j}(x):=(x+d_{j})/n$ for $j=1,\ldots,m.$ Let $\mathcal{C}_{0}:=[0,1],$
and inductively\[
\mathcal{C}_{k+1}:=\bigcup_{j=1}^{m}S_{j}(\mathcal{C}_{k}).\]
Then $\mathcal{C}_{k}=\{0.x_{1}x_{2}\ldots\mid x_{j}\in\mathcal{D}\text{ for }j\leq k\}$
is the set of real numbers in the interval $[0,1]$ that admit an
$n-$ary representation whose first $k$ digits are chosen from the
digit set $\mathcal{D}.$ Consequently, $\mathcal{C}=\bigcap_{k=0}^{\infty}\mathcal{C}_{k}.$
Furthermore, for each $k,$ the set $\mathcal{C}_{k}$ consists of
$m^{k}$ closed intervals each of length $1/n^{k}$ and these intervals
are $n-$ary intervals with disjoint interiors.

\subsubsection{Retained/deleted intervals construction of $\mathcal{C}$ }

If $I=[a,b]$ is a closed interval, we can consider the partition,\[
I_{j}:=\left[a+\frac{j-1}{n}(b-a),a+\frac{j}{n}(b-a)\right],\]
$j=1,2,\ldots,n,$ of $I$ into $n$ closed subintervals of equal
length. The subset of $I,$ obtained by retaining the intervals in
the partition corresponding to digits in $\mathcal{D},$ that is the
set \[
\bigcup_{j=1}^{m}\left[a+\frac{d_{j}}{n}(b-a),a+\frac{d_{j}+1}{n}(b-a)\right]\]
is a \emph{refinement} of the interval $[a,b].$ With this terminology
$\mathcal{C}_{k+1}$ is obtained from $\mathcal{C}_{k}$ be refining
each interval in $\mathcal{C}_{k}.$

\subsection{Investigating $\mathcal{C}\cap\left(\mathcal{C}+x\right)$\label{sub:Investigating}}

In the remainder of this section we will assume $\mathcal{D}$ satisfies
the \emph{separation condition} $2\leq d_{k+1}-d_{k}$ for $k=1,2,\ldots,m-1.$
Note that \[
\mathcal{C}\cap\left(\mathcal{C}+x\right)=\bigcap_{k=0}^{\infty}\left(\mathcal{C}_{k}\cap\left(\mathcal{C}_{k}+x\right)\right).\]
For $x=0.x_{1}x_{2}\cdots$ let $\left\lfloor x\right\rfloor _{k}$
denote the truncation to the first $k$ places, that is, \[
\left\lfloor x\right\rfloor _{k}:=0.x_{1}x_{2}\cdots x_{k}.\]
If $x$ admit a finite $n-$ary representation $\left\lfloor x\right\rfloor _{k}$
depends on which $n-$ary representation is chosen. We will consider
$\mathcal{C}_{k}\cap(\mathcal{C}_{k}+\left\lfloor x\right\rfloor _{k})$
in place of $\mathcal{C}_{k}\cap\left(\mathcal{C}_{k}+x\right),$
since both $\mathcal{C}_{k}$ and $\mathcal{C}_{k}+\left\lfloor x\right\rfloor _{k}$
consists of $n-$ary intervals of lengths $1/3^{k}.$ 

Below, \emph{an interval in} $\mathcal{C}_{j},$ is short for an $n-$ary
interval in $\mathcal{C}_{j}$ of length $1/n^{j}$, that is one of
the interval obtained by applying the refinement process. A similar
convention applies to the term \emph{an interval in} $\mathcal{C}_{j}+y.$ 

We will investigate how $\mathcal{C}_{k+1}\cap(\mathcal{C}_{k+1}+\left\lfloor x\right\rfloor _{k+1})$
is related to $\mathcal{C}_{k}\cap(\mathcal{C}_{k}+\left\lfloor x\right\rfloor _{k})$
for $k\geq0.$ Recall that $\mathcal{C}_{k+1}$ is obtainied from
$\mathcal{C}_{k}$ by refining each interval in $\mathcal{C}_{k}.$
Consequently, $\mathcal{C}_{k+1}+\left\lfloor x\right\rfloor _{k+1}$
is obtained from $\mathcal{C}_{k}+\left\lfloor x\right\rfloor _{k}$
by refining each interval in $\mathcal{C}_{k}+\left\lfloor x\right\rfloor _{k}$
and then translating the resulting intervals to the right by $x_{k+1}/n^{k+1}.$ 

Let $I$ be one of the intervals in $\mathcal{C}_{k}.$ We will consider
what happens to $I$ as we \emph{transition} from $\mathcal{C}_{k}\cap(\mathcal{C}_{k}+\left\lfloor x\right\rfloor _{k})$
to $\mathcal{C}_{k+1}\cap(\mathcal{C}_{k+1}+\left\lfloor x\right\rfloor _{k+1}).$
Since $\left\lfloor x\right\rfloor _{k}$ is an integral multiple
of $1/n^{k},$ the intervals in $\mathcal{C}_{k}+\left\lfloor x\right\rfloor _{k}$
either coincides with intervals in $\mathcal{C}_{k},$ they have one
or both endpoints in common with intervals in $\mathcal{C}_{k},$
or are at least $1/n^{k}$ units away from any interval in $\mathcal{C}_{k}.$
Hence there are four possibilities for the interval $I$ to consider. 
\begin{itemize}
\item $I$ is in the \emph{interval case}, if there an interval $J$ in
$\mathcal{C}_{k}+\left\lfloor x\right\rfloor _{k}$ such that $I=J.$
\item $I$ is in the \emph{potential interval case}, if there is an interval
$J$ in $\mathcal{C}_{k}+\left\lfloor x\right\rfloor _{k}$ such that
the left-hand endpoint of $J$ is the right-hand endpoint of $I.$
\item $I$ is in the \emph{potentially empty case}, if there is an interval
$J$ in $\mathcal{C}_{k}+\left\lfloor x\right\rfloor _{k}$ such that
the left-hand endpoint of $J$ is the right-hand endpoint of $I.$
\item $I$ is in the \emph{empty case}, if $I$ does not intersect any interval
in $\mathcal{C}_{k}+\left\lfloor x\right\rfloor _{k}.$ 
\end{itemize}
It is possible for $I$ to be simultaneously be in the potential interval
case and the potentially empty case. Due to the separation condition
it is not possible for an interval in $\mathcal{C}_{k}$ to simultaneously
be in the interval case and the potential interval case, or simultaneously
in the interval case and the potentially empty case. 

For any $x$ \begin{equation}
\mathcal{C}\cap\left(\mathcal{C}+x\right)=\bigcap_{k=0}^{\infty}\bigcup I,\label{eq:IntervalDescription}\end{equation}
where the union is over the intervals in $\mathcal{C}_{k}$ that are
not in the empty case.

We will describe the four cases above in more detail, under the assumption
that  $x=0.x_{1}x_{2}\cdots$ does not terminate in repeating $0$'s
or in repeating $n-1$'s, equivalently, $0<x-\left\lfloor x\right\rfloor _{k}<1/n^{k}$
for all $k,$ in particular, we will see that we can exclude the intervals
in $\mathcal{C}_{k}$ that are in the potentially empty case from
the union in (\ref{eq:IntervalDescription}).

\subsubsection{Suppose $I$ is in the interval case}

Let $J$ be the interval in $\mathcal{C}_{k}+\left\lfloor x\right\rfloor _{k}$
such that $I=J.$ Then $I\cap\left(J+x-\left\lfloor x\right\rfloor _{k}\right)$
is an interval of length $\frac{1}{n^{k}}-\left(x-\left\lfloor x\right\rfloor _{k}\right)$
hence the refinement/translation process applied to the intervals
$I$ and $J$ may lead to points in $\mathcal{C}\cap\left(\mathcal{C}+x\right).$

\subsubsection{Suppose $I$ is in the potential interval case}

Let $J$ be the interval in $\mathcal{C}_{k}+\left\lfloor x\right\rfloor _{k}$
such that is the left-hand endpoint of $J$ is the right-hand endpoint
of $I.$ Then $I\cap\left(J+x-\left\lfloor x\right\rfloor _{k}\right)$
is an interval of length $x-\left\lfloor x\right\rfloor _{k}$ hence
the refinement/translation process applied to the intervals $I$ and
$J$ may lead to points in $\mathcal{C}\cap\left(\mathcal{C}+x\right).$

\subsubsection{Suppose $I$ is in the potentially empty case}

Let $J$ be the interval in $\mathcal{C}_{k}+\left\lfloor x\right\rfloor _{k}$
such that the right-hand endpoint of $J$ is the left-hand endpoint
of $I.$ Since $0<x-\left\lfloor x\right\rfloor _{k}$ we have $I\cap\left(J+x-\left\lfloor x\right\rfloor _{k}\right)=\emptyset$,
so this intersection does not lead to points in $\mathcal{C}\cap\left(\mathcal{C}+x\right).$

\subsubsection{Suppose $I$ is in the empty case}

Then $I\cap\left(\mathcal{C}_{k}+\left\lfloor x\right\rfloor _{k}\right)=\emptyset.$
Since the interval $I$ is at least $1/n^{k}$ units away from any
interval $\mathcal{C}_{k}+\left\lfloor x\right\rfloor _{k}$ in and
$x-\left\lfloor x\right\rfloor _{k}<1/n^{k}$ we have $I\cap\left(\mathcal{C}_{k}+x\right)=\emptyset.$
So $I$ does not contribute points to $\mathcal{C}\cap\left(\mathcal{C}+x\right).$ 
\begin{rem}
If $I$ is any interval in $\mathcal{C}_{k}$ and $J$ is any interval
in $\mathcal{C}_{k}+\left\lfloor x\right\rfloor _{k}$ disjoint from
$I,$ then $I$ and $J+x-\left\lfloor x\right\rfloor _{k}$ are also
disjoint. The reason for this is that the separation condition $d_{j+1}-d_{j}\geq2$
implies the distance between intervals $I$ and $J$ is an integral
multiple of $1/n^{k}$ and $0\leq x-\left\lfloor x\right\rfloor _{k}<1/n^{k}.$
Hence, when we consider $\mathcal{C}\cap\left(\mathcal{C}+x\right),$
we can ignore intervals $J$ that do not occur in the first two cases
above. (The first three cases, if we allow $x$ to admit a finite
$n-$ary representation.)
\end{rem}

\subsection{A Description of $\mathcal{F}$}

The following description of $\mathcal{F}$ is useful below. Let $x$
be an element of $\mathcal{F}$ such that $0<x<1.$

\subsubsection{Suppose $x$ does not have a finite $n-$ary expansion}

Then $0<x-\left\lfloor x\right\rfloor _{k}<1/n^{k}$ for all $k,$
hence we can apply the analysis at the end of sub-section \ref{sub:Investigating}.
In particular, for all $k,$ at least one of the intervals in $\mathcal{C}_{k}$
will either be in the interval case or in the potential interval case.

\subsubsection{Suppose $x$ has a finite $n-$ary expansion}

Then we can write $x=0.x_{1}x_{2}\cdots x_{k}$ where $x_{k}\neq0.$
Then \[
\mathcal{C}\cap\left(\mathcal{C}+x\right)=\bigcap_{j=k}^{\infty}\left(\mathcal{C}_{j}\cap\left(\mathcal{C}_{j}+\left\lfloor x\right\rfloor _{k}\right)\right),\]
since $x=\left\lfloor x\right\rfloor _{k}.$ Consequently, if one
of the intervals in $\mathcal{C}_{k}$ is in the interval case, then
repeated refinement of that interval leads to a subset of $\mathcal{C}\cap\left(\mathcal{C}+x\right)$
that is similar to $\mathcal{C}.$ 

On the other hand, if one of the intervals $I$ in $\mathcal{C}_{k}$
is in the potential interval case or in the potentially empty case,
then there is an interval $J$ in $\mathcal{C}_{k}+x$ such that $I\cap J$
contains exactly one point, $y$ say. If $d_{m}=n-1,$ then $y$ will
be contained in the intersection of the refinements of $I$ and $J.$
Hence $y$ will be a point in $\mathcal{C}\cap\left(\mathcal{C}+x\right)$
and any other point in $\mathcal{C}\cap\left(\mathcal{C}+x\right)$
will be at least $1/n^{k}$ units away from $y.$ On the other hand,
if $d_{m}<n-1,$ then the refinements of $I$ and $J$ will not intersect.
In particular, $y$ will not be a point in $\mathcal{C}\cap\left(\mathcal{C}+x\right).$ 
\begin{rem}
The description above shows that, if $x$ has a finite $n-$ary expansion,
then \[
\mathcal{C}\cap\left(\mathcal{C}+x\right)=E\cup F,\]
where $E$ is a finite, perhaps empty, union of sets similar to $\mathcal{C}$
and $F$ is a finite, perhaps empty, set. Hence, if $x$ has a finite
$n-$ary expansion, then the dimension $\mathcal{C}\cap\left(\mathcal{C}+x\right)$
is either $0$ or $\log m/\log n.$ Consequently, to prove Theorem
\ref{thm:non-uniform} we must consider $x$ that do not have finite
$n-$ary expansions. 
\end{rem}

\section{A Stepping Stone\label{sec:A-Stepping-Stone}}

The following provides the key step in the proofs of Theorem \ref{thm:non-uniform}
and Theorem \ref{thm:uniform} and is, perhaps, of independent interest.
\begin{prop}
\label{pro:Main-Step} If $\mathcal{D}$ satisfies the separation
condition $d_{j+1}-d_{j}\geq2$ for all $j=1,2,\ldots,m-1,$ then
given any $0\leq\alpha\leq1,$ there is an $x$ in $\mathcal{F}$
such that $\mathcal{C}\cap(\mathcal{C}+x)$ has Minkowski and Hausdorff
dimension equal to $\alpha\log_{n}m.$ This $x$ may be chosen not
to admit a terminating $n-$ary representation. 
\end{prop}
Let $0\leq\alpha\leq1$ be given. The proof is completed in two steps.
First we use the transition process to construct an $x$ such that
$\mathcal{C}\cap(\mathcal{C}+x)$ has Minkowski dimension $\alpha\log_{n}m.$
Then we show that for this $x$ the set $\mathcal{C}\cap(\mathcal{C}+x)$
also has Hausdorff dimension $\alpha\log_{n}m.$

\subsection{Construction of $x$ using the Minkowski dimension as a guide}

We begin the refinement process in the interval case $\mathcal{C}_{0}\cap\left(\mathcal{C}_{0}+0\right).$
The idea of the proof is, if $x_{j+1}=0,$ then transitioning from
$\mathcal{C}_{j}\cap(\mathcal{C}_{j}+\left\lfloor x\right\rfloor _{j})$
to $\mathcal{C}_{j+1}\cap(\mathcal{C}_{j+1}+\left\lfloor x\right\rfloor _{j+1})$
multiplies the number of interval cases by $m,$ and if $x_{j+1}=d_{m}$
the transition multiplies the number of interval cases by one. In
either case no potential interval cases or potentially empty cases
appear. 

Let $h_{j}:=\left[j\alpha\right],$ then $h_{j}$ is a positive integer
such that $h_{j}\leq j\alpha<1+h_{j},$ and consequently, $h_{j}/j\to\alpha$
as $j\to\infty.$ Since $0\leq\alpha\leq1$ we have $h_{j}\leq h_{j+1}\leq1+h_{j}.$
Suppose $0<\alpha<1.$ For $j\geq1$ set\[
x_{j}=\begin{cases}
d_{m} & \text{if }h_{j}=h_{j-1}\\
0 & \text{if }h_{j}=1+h_{j-1}\end{cases}.\]
Then the number of interval cases in $\mathcal{C}_{j}\cap(\mathcal{C}_{j}+\left\lfloor x\right\rfloor _{j})$
is $m^{h_{j}}.$ Since $\mathcal{C}\cap\left(\mathcal{C}+x\right)$
is a subset of $\mathcal{C}_{j}\cap(\mathcal{C}_{j}+\left\lfloor x\right\rfloor _{j})$
this provides an upper bound for the number of interval of length
$1/n^{j}$ needed to cover $\mathcal{C}\cap\left(\mathcal{C}+x\right):$
\begin{equation}
\mathcal{N}_{j}\left(\mathcal{C}\cap\left(\mathcal{C}+x\right)\right)\leq m^{h_{j}}.\label{eq:upper-bound}\end{equation}
To calculate the Minkowski dimension of $\mathcal{C}\cap\left(\mathcal{C}+x\right)$
it remains to check that any interval case in $\mathcal{C}_{j}\cap(\mathcal{C}_{j}+\left\lfloor x\right\rfloor _{j})$
leads to points in $\mathcal{C}\cap\left(\mathcal{C}+x\right)$ so
that the upper bound (\ref{eq:upper-bound}) for $\mathcal{N}_{j}(\mathcal{C}\cap\left(\mathcal{C}+x\right))$
is also a lower bound. But by the refinement process each interval
in $\mathcal{C}_{j}\cap(\mathcal{C}_{j}+\left\lfloor x\right\rfloor _{j})$
transitions to one or $m$ sub-intervals in $\mathcal{C}_{j+1}\cap(\mathcal{C}_{j+1}+\left\lfloor x\right\rfloor _{j+1}).$
Hence, it follows from the Nested Interval Theorem that each interval
in $\mathcal{C}_{j}\cap(\mathcal{C}_{j}+\left\lfloor x\right\rfloor _{j})$
has infinitely many points in common with $\mathcal{C}\cap\left(\mathcal{C}+x\right).$
Using $h_{j}/j\to\alpha,$ we conclude \begin{equation}
\frac{\log\mathcal{N}_{j}\left(\mathcal{C}\cap\left(\mathcal{C}+x\right)\right)}{j\log n}=\frac{h_{j}\log m}{j\log n}\to\alpha\frac{\log m}{\log n}\label{eq:limit}\end{equation}
as $j\to\infty.$ 

We can change some of the digits $x_{j}=0$ to $x_{j}=d_{m}$ or visa
versa, as long as the limit (\ref{eq:limit}) remains unchanged. That
is, we can make changes of this nature on a sparse set of $j$'s.
In particular, if necessary, we can ensure that $x$ neither terminates
in repeating $0$'s nor in repeating $d_{m}$'s. In particular, this
observation allows us to deal with the cases $\alpha=0$ and $\alpha=1,$
using arguments presented above. The details are left for the reader.

\subsection{Hausdorff dimension\label{sub:Hausdorff-dimension}}

It is not immediate that the Hausdorff dimension of $\mathcal{C}\cap\left(\mathcal{C}+x\right)$
equals its Minkowski dimension because the set $\mathcal{C}\cap\left(\mathcal{C}+x\right)$
need not be self-similar, see Section \ref{sec:Open-Questions}. Consequently,
it remains to check that the Hausdorff dimension of $\mathcal{C}\cap\left(\mathcal{C}+x\right)$
is bounded below by $\alpha\log_{n}m.$ The argument below is inspired
by an argument due to Eggleston \cite{Egg49}.

Let $\mathcal{A}$ be the collection of $n-$ary intervals introduced
as part of the construction of $x.$ Then, the subcollection $\mathcal{A}_{j}$
of $n-$ary intervals in $\mathcal{A}$ of length $1/n^{j}$ contains
$m^{h_{j}}$ members, where $h_{0}=0,$ and $h_{j}/j\to\alpha$ as
$j\to\infty.$ And, by equality in (\ref{eq:upper-bound}), each interval
in $\mathcal{A}_{j}$ refines to $m^{h_{j+1}-h_{j}}$ intervals in
$\mathcal{A}_{j+1}.$ Let $\beta<\alpha\log_{n}m.$ Then, $m^{h_{j}/j}\to m^{\alpha}$
implies $\sum_{j=0}^{\infty}n^{j\beta}/m^{h_{j}}<\infty.$  Let $N$
be an integer such that \begin{equation}
\sum_{j=N}^{\infty}\frac{n^{j\beta}}{m^{h_{j}}}<\frac{1}{2}.\label{eq:1}\end{equation}
Let $\mathcal{U}$ be a collection of $n-$ary intervals covering
$\mathcal{C}\cap\left(\mathcal{C}+x\right)$ and whose lengths are
at most $1/n^{N}.$ We will show that $\sum_{I\in\mathcal{U}}|I|^{\beta}\geq1.$
Consequently, the Hausdorff dimension of $\mathcal{C}\cap\left(\mathcal{C}+x\right)$
is bounded below by $\beta.$ Since $\beta<\alpha\log_{n}m$ is arbitrary,
it follows that the Hausdorff dimension of $\mathcal{C}\cap\left(\mathcal{C}+x\right)$
is bounded below by $\alpha\log_{n}m.$ The proof is by contradiction.
Suppose \begin{equation}
\sum_{I\in\mathcal{U}}|I|^{\beta}<1.\label{eq:2}\end{equation}
Remove the intervals from $\mathcal{U}$ that do not intersect $\mathcal{C}\cap\left(\mathcal{C}+x\right).$
If an interval $I$ in $\mathcal{U}$ is not in $\mathcal{A},$ but
shares an endpoint with an interval $J$ in $\mathcal{A},$ replace
$I$ by $J.$ Making these changes to $\mathcal{U}$ will not increase
the sum in (\ref{eq:2}), hence we can assume $\mathcal{U}$ is a
subset of $\mathcal{A}.$ Let $\mathcal{U}_{j}$ be the intervals
in $\mathcal{U}$ of length $1/n^{j}.$ By convergence of the sum
in (\ref{eq:2}) the set $\mathcal{U}_{j}$ is finite. If $\#\mathcal{U}_{j}$
denotes the number of intervals in $\mathcal{U}_{j},$ then \[
\left(\#\mathcal{U}_{j}\right)\frac{1}{n^{j\beta}}=\sum_{I\in\mathcal{U}_{j}}|I|^{\beta}<1\]
by (\ref{eq:2}). Consequently, \begin{equation}
\sum_{I\in\mathcal{U}_{j}}|I|=\left(\#\mathcal{U}_{j}\right)\frac{1}{n^{j}}<\frac{n^{j\beta}}{n^{j}}.\label{eq:3}\end{equation}
If $j\leq k,$ then any interval in $\mathcal{U}_{j}$ is an interval
in $\mathcal{A}_{j}$ and refines to $m^{h_{k}-h_{j}}$ intervals
in $\mathcal{A}_{k}$. Hence every interval in $\mathcal{U}_{j}$
covers exactly $m^{h_{k}-h_{j}}$ intervals in $\mathcal{A}_{k}.$
So, if $\mathcal{B}_{j,k}$ is the intervals in $\mathcal{A}_{k}$
covered by intervals in $\mathcal{U}_{j},$ then \[
\sum_{I\in\mathcal{B}_{j,k}}|I|=\frac{m^{h_{k}-h_{j}}}{n^{k-j}}\sum_{J\in\mathcal{U}j}|J|.\]
The factor $1/n^{k-j}$ appears since $|I|=1/n^{k}$ and $|J|=1/n^{j}.$
So by (\ref{eq:3}) \[
\sum_{I\in\mathcal{B}_{j,k}}|I|<\frac{m^{h_{k}-h_{j}}}{n^{k-j}}\frac{n^{j\beta}}{n^{j}}.\]
Hence \[
\sum_{j=N}^{k}\sum_{I\in\mathcal{B}_{j,k}}|I|<\frac{m^{h_{k}}}{n^{k}}\sum_{j=N}^{k}\frac{n^{j\beta}}{m^{h_{j}}}<\frac{1}{2}\frac{m^{h_{k}}}{n^{k}}=\frac{1}{2}\sum_{I\in\mathcal{A}_{k}}|I|.\]
Where the last inequality used (\ref{eq:1}). Consequently, there
are intervals in $\mathcal{A}_{k}$ disjoint from $\cup_{j=N}^{k}\cup_{I\in\mathcal{U}_{j}}I.$
Let $\mathcal{H}_{k}$ be the union of the intervals in $\mathcal{A}_{k}$
that are disjoint for $\cup_{j=N}^{k}\cup_{I\in\mathcal{U}_{j}}I.$
The $\mathcal{H}_{k+1}$ is a subset of $\mathcal{H}_{k},$ hence,
by compactness, $\cap_{k=N}^{\infty}\mathcal{H}_{k}$ is non-empty.
Any point in $\cap_{k=N}^{\infty}\mathcal{H}_{k}$ is a point in $\mathcal{C}\cap\left(\mathcal{C}+x\right)$
not covered by any interval in $\mathcal{U}.$ Contradicting that
$\mathcal{U}$ is a cover of $\mathcal{C}\cap\left(\mathcal{C}+x\right).$ 

We have shown $\dim_{\mathrm{M}}\mathcal{C}\cap\left(\mathcal{C}+x\right)\leq\dim_{\mathrm{H}}\mathcal{C}\cap\left(\mathcal{C}+x\right).$
Consequently, the Hausdorff and Minkowski dimensions of $\mathcal{C}\cap\left(\mathcal{C}+x\right)$
are equal.

\section{Proof of Theorem \ref{thm:non-uniform}\label{sec:Proof-of-Theorem Main}}

Since $d_{m}<n-1,$ the largest element of $\mathcal{C},$ that is
$0.d_{m}d_{m}d_{m}\cdots,$ is $<1.$ Hence, if $0\leq y\leq1$ is
such that $\mathcal{C}\cap\left(\mathcal{C}+y\right)$ is non-empty,
then $y<1.$ 

Let $0\leq y<1$ be such that $\mathcal{C}\cap\left(\mathcal{C}+y\right)$
is non-empty. Let $\varepsilon>0$ be given. Write $y=0.y_{1}y_{2}\cdots.$
Pick $k$ so large that $1/n^{k}<\varepsilon.$ Let $x_{j}=y_{j}$
for $j=1,2,\ldots,k.$ Then no matter how we determine $x_{j}$ for
$j>k,$ we have $|x-y|<\varepsilon,$ where $x=0.x_{1}x_{2}\cdots.$
We will ensure that $x$ does not have a terminating $n-$ary expansion.

\subsection{Interval Cases}

Suppose there is an interval in $\mathcal{C}_{k}\cap(\mathcal{C}_{k}+\left\lfloor x\right\rfloor _{k}),$
then at least one interval in $\mathcal{C}_{k}$ is in the interval
case. Setting $x_{k+1}=0,$ each interval in $\mathcal{C}_{k}\cap(\mathcal{C}_{k}+\left\lfloor x\right\rfloor _{k})$
transitions to $m$ intervals in $\mathcal{C}_{k+1}\cap(\mathcal{C}_{k+1}+\left\lfloor x\right\rfloor _{k+1})$
and, using that $d_{m}<n-1$ we see that all potential interval cases
and potentially empty cases transitions to empty cases. Hence we now
have interval cases only. The transition to interval cases only, is
illustrated in Figure \ref{fig:Potential-zero}, in the case where
$n=7$ and $\mathcal{D}=\left\{ 0,3,5\right\} .$ The two long lines
represents the interval $I$ in $\mathcal{C}_{k}$ and the interval
$J$ in $\mathcal{C}_{k}+\left\lfloor x\right\rfloor _{k}.$ Depending
on whether the line on the left is $I$ or $J$ Figure \ref{fig:Potential-zero}
illustrates the potentially empty or the potential interval case.
The heavier lines are the refinements of the two intervals. For clarity
one line is slightly elevated compared to the other line. 

\begin{figure}[h]
\includegraphics{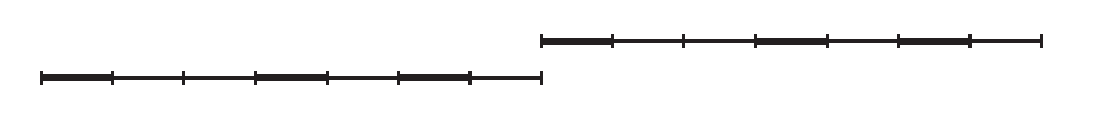}

\caption{\label{fig:Potential-zero}Setting $x_{k+1}=0$ and refining the intervals}

\end{figure}
We are left with a finite number of interval cases in $\mathcal{C}_{k+1}.$
We can apply Proposition \ref{pro:Main-Step} to, a self-similar copy
of, one of these intervals cases to arrive at the $x_{j}$ for $j>k+1.$
Since all the interval cases essentially are identical, it follows
that $\mathcal{C}\cap\left(\mathcal{C}+x\right)$ is a finite union
of sets whose Minkowski and Hausdorff dimensions both equal $\alpha\log_{n}m.$

\subsection{Remaining Cases}

Suppose $\mathcal{C}_{k}\cap(\mathcal{C}_{k}+\left\lfloor x\right\rfloor _{k})$
does not contain an interval, but does contain potential interval
cases. Let $I$ be an interval in $\mathcal{C}_{k}$ and $J$ an interval
in $\mathcal{C}_{k}+\left\lfloor x\right\rfloor _{k}$ such that the
right endpoint of $J$ coincides with the left endpoint of $I.$ Setting
$x_{k+1}=n-d_{m},$ then translates the right-most interval in the
refinement of $J$ onto the left-most interval in the refinement of
$I.$ Hence $\mathcal{C}_{k+1}\cap(\mathcal{C}_{k+1}+\left\lfloor x\right\rfloor _{k+1})$
contains an interval and we can now repeat the previous argument. 

Finally, suppose $\mathcal{C}_{k}\cap(\mathcal{C}_{k}+\left\lfloor x\right\rfloor _{k})$
does not contain an interval nor a potential interval case, but does
contain a potentially empty case. Then $\mathcal{C}_{k+1}\cap(\mathcal{C}_{k+1}+\left\lfloor x\right\rfloor _{k+1})$
will be empty even if $x_{k+1}=0,$ because $d_{m}<n-1.$ Contradicting
that $\mathcal{C}\cap\left(\mathcal{C}+y\right)$ is non-empty.

\section{Proof of Theorem \ref{thm:uniform}\label{sec:Proof-of-Theorem-uniform}}

If $d_{m}<n-1,$ this is a special case of Theorem \ref{thm:non-uniform},
hence we may assume that $d_{m}=n-1.$ 

The proof is similar to the proof of Theorem \ref{thm:non-uniform},
we will comment on the differences. If there is an interval in $\mathcal{C}_{k}\cap(\mathcal{C}_{k}+\left\lfloor x\right\rfloor _{k}),$
then there are no potential interval cases or potentially empty cases.
Hence there is no reason to begin by setting $x_{k+1}=0.$ The rest
of the proof in the interval case is unchanged. 

Suppose $\mathcal{C}_{k}\cap(\mathcal{C}_{k}+\left\lfloor x\right\rfloor _{k})$
does not contain an interval, but does contain potential interval
cases. In this case the argument remains unchanged. Note that $x_{k+1}=1.$

Finally, suppose $\mathcal{C}_{k}\cap(\mathcal{C}_{k}+\left\lfloor x\right\rfloor _{k})$
does not contain an interval nor a potential interval case, but does
contain a potentially empty case. In this case the argument presented
above does not work. Note $y_{j}=0$ for all $j>k,$ because otherwise
$\mathcal{C}\cap\left(\mathcal{C}+y\right)$ would be empty. Also,
$y>0,$ since we are not in the interval case. Let $l$ be the largest
subscript such that $y_{l}\neq0.$ Let $z=0.z_{1}z_{2}\cdots,$ where
$z_{j}=y_{j}$ when $j<l,$ $z_{l}=y_{l}-1,$ and $z_{j}=n-1$ when
$l<j.$ Then $\left\lfloor z\right\rfloor _{k}\to y$ and $\mathcal{C}_{k}\cap(\mathcal{C}_{k}+\left\lfloor z\right\rfloor _{k})$
is in the interval case for all $k\geq l.$ Hence, we can, once again,
apply the interval case argument.

\section{Geometry of $\mathcal{F}.$ \label{sec:Geometry-of-F}}

Let $m<n$ be positive integers and let $\mathcal{D}=\{d_{1},d_{2},\ldots,d_{m}\}$
be a set of at least two integers, such that $0=d_{1}<d_{2}<\cdots<d_{m}<n$
and let $\mathcal{C}$ be the corresponding deleted digits Cantor
set. The set $\mathcal{F}$ of all $x$ in the interval $[0,1],$
such that $\mathcal{C}\cap\left(\mathcal{C}+x\right)$ is non-empty
plays a starring role in Theorem \ref{thm:non-uniform}. The main
purpose of this section is to investigate the geometry of $\mathcal{F}.$ 

Since $\mathcal{C}$ is compact, $\mathcal{F}$ is also compact: 
\begin{lem}
$\mathcal{F}$ is compact and non-empty. \end{lem}
\begin{proof}
Clearly $0$ is in $\mathcal{F}$ and $\mathcal{F}$ is bounded, since
$\mathcal{C}$ is bounded. Hence it is sufficient to show that $\mathcal{F}$
is closed. Let $x>0.$ Suppose $\mathcal{C}\cap\left(\mathcal{C}+x\right)$
is empty. By compactness of $\mathcal{C},$\[
\varepsilon:=\mathrm{dist}\left(\mathcal{C},\mathcal{C}+x\right)>0.\]
Hence, $\mathcal{C}\cap\left(\mathcal{C}+y\right)$ is empty when
$|x-y|<\varepsilon.$ Consequently, the complement of $\mathcal{F}$
in the interval $[0,\infty)$ is an open set. 
\end{proof}
We begin by showing that $ $$\mathcal{G}:=\left(-\mathcal{F}\right)\cup\mathcal{F}$
is the attractor for a set of similarity transformations. 

Let $\Delta:=\mathcal{D}-\mathcal{D}=\left\{ d-e\mid d,e\in\mathcal{D}\right\} $
and\begin{equation}
\sigma_{\delta}(x):=\frac{x+\delta}{n}\label{eq:SigmaF}\end{equation}
for $\delta\in\Delta.$ Since $0$ is in $\mathcal{D},$ we have $\pm\mathcal{D}\subseteq\Delta.$
Since $\sigma_{d_{j}}=S_{j}$ for $j=1,2,\ldots,m.$ This family of
contractive similarities is closely related to the similarities used
above to generate $\mathcal{C}.$ 
\begin{lem}
$\mathcal{G}$ is the unique non-empty compact set invariant under
the contractions $\sigma_{\delta},$ $\delta\in\Delta.$ \end{lem}
\begin{proof}
Note $\mathcal{G}$ is the set of real numbers $x$ such that $\mathcal{C}\cap\left(\mathcal{C}+x\right)$
is non-empty. Consequently, $\mathcal{G}=\mathcal{C}-\mathcal{C}=\left\{ s-t\mid s,t\in\mathcal{C}\right\} .$
Then the self-similarity construction of $\mathcal{C}$ leads to\begin{align*}
\mathcal{C}-\mathcal{C} & =\bigcup_{d,e\in\mathcal{D}}\sigma_{d}\left(\mathcal{C}\right)-\sigma_{e}\left(\mathcal{C}\right)\\
 & =\bigcup_{\delta\in\Delta}\sigma_{\delta}\left(\mathcal{C}-\mathcal{C}\right)\end{align*}
where the second equality used $\sigma_{d}(x)-\sigma_{e}(y)=\sigma_{d-e}(x-y).$ 
\end{proof}
This leads to a characterization of when $\mathcal{F}$ is an interval: 
\begin{prop}
\label{pro:F-interval}Write $\Delta=\left\{ \delta_{j}\mid j=1,2,\ldots,M\right\} ,$
where $\delta_{j}<\delta_{j+1}$ for $j=1,2,\ldots,M-1.$ Then $\mathcal{F}$
is an interval iff\[
2d_{m}\geq(n-1)\left(\delta_{j+1}-\delta_{j}\right)\]
for all $j=1,2,\ldots,M-1.$ If $\mathcal{F}$ is an interval, then
$\mathcal{F}=\left[0,d_{m}/(n-1)\right].$\end{prop}
\begin{proof}
The largest element of $\mathcal{C}$ is $d_{m}/(n-1),$ hence the
smallest interval containing $\mathcal{G}$ is\[
I:=\left[-d_{m}/(n-1),d_{m}/(n-1)\right].\]
So $\mathcal{G}$ is an interval iff $\mathcal{G}=I.$ By construction
of $I$ we have $\bigcup_{\delta\in\Delta}\sigma_{\delta}(I)\subseteq I.$
Since $\sigma_{\delta}\left(\mathcal{G}\right)\subseteq\sigma_{\delta}(I),$
each of the intervals $\sigma_{\delta}(I),\delta\in\Delta$ has points
in common with $\mathcal{G}.$ Consequently, $\mathcal{G}$ is an
interval iff the intervals $\sigma_{\delta_{j}}(I)$ and $\sigma_{\delta_{j+1}}(I)$
overlap for all $j=1,2,\ldots,M-1.$ Hence $\mathcal{G}$ is an interval
iff \[
\sigma_{\delta_{j}}\left(d_{m}/(n-1)\right)\geq\sigma_{\delta_{j}+1}\left(-d_{m}/(n-1)\right)\]
for all $j=1,2,\ldots,M-1.$ By (\ref{eq:SigmaF}) this is equivalent
to the condition listed above. \end{proof}
\begin{cor}
\label{cor:interval} If $\mathcal{D}$ consists of even integers,
then $\mathcal{F}$ is an interval if and only if \[
\mathcal{D}-\mathcal{D}=\left\{ -n+1,\ldots,-2,0,2,\ldots,n-1\right\} .\]
In the affirmative case $n$ is odd and $\mathcal{F}=[0,1].$ \end{cor}
\begin{example}
It is a simple consequence of Corollary \ref{cor:interval} that $\mathcal{F}=[0,1],$
if $n=7$ and $\mathcal{D}=\{0,2,6\}.$ 
\begin{example}
Suppose $d_{j}=2(j-1)$ for $j=1,2,\ldots,m$ and $2(m-1)=n-1,$ then
$\mathcal{F}=[0,1]$ by Corollary \ref{cor:interval}. Let $f(t):=\underline{\dim}_{\mathrm{M}}\left(\mathcal{C}\cap\left(\mathcal{C}+t\right)\right)$
be the lower Minkowski dimension of $\mathcal{C}\cap(\mathcal{C}+t).$
By Theorem \ref{thm:uniform} $f$ maps any subinterval of $[0,1]$
onto the interval $[0,\log m/\log n].$ In particular, $f$ is discontinuous
at every point in the interval $[0,1].$ 
\end{example}
\end{example}
If there is a bounded open interval $I$ such that 

\begin{equation}
\bigcup_{\delta\in\Delta}\sigma_{\delta}(I)\subseteq I\text{ \quad(disjoint union)}\label{eq:open-interval-condition}\end{equation}
then $\mathcal{G}$ is a subset of the closure of $I.$ Hence, it
follows from the proof of Proposition \ref{pro:F-interval}, that
there is a bounded \emph{open} interval $I,$ such that (\ref{eq:open-interval-condition})
if and only if \begin{equation}
2d_{m}\leq(n-1)\left(\delta_{j+1}-\delta_{j}\right)\label{eq:open-set-condition}\end{equation}
for all $j=1,2,\ldots,M-1.$ In the affirmative case\[
I=\left(-d_{m}/(n-1),d_{m}/(n-1)\right).\]
In particular, (\ref{eq:open-set-condition}) implies that $\sigma_{\delta},$
$\delta\in\Delta$ satisfies the open set condition. Now $\mathcal{D}$
has $m$ elements and $\pm\mathcal{D}\subseteq\Delta,$ so, since
each digit in $\mathcal{D}$ is non-negative, $\Delta$ has at least
$2m-1$ elements. Hence the sparcity condition (\ref{eq:open-set-condition})
implies \[
\dim_{\mathrm{H}}\mathcal{F}=\log_{n}\#\Delta>\log_{n}\#\mathcal{D}=\dim_{\mathrm{H}}\mathcal{C}\]
by Hutchinson \cite{Hut81}, \cite{Fal85}. 

The following explores the structure of $\mathcal{F}$ under the conditions
of Theorem \ref{thm:uniform}. These conditions imply (\ref{eq:open-set-condition}). 
\begin{prop}
\label{pro:DescriptionOfF} If there is an integer $h>1$ dividing
all the digits in $\mathcal{D},$ then $\mathcal{F}$ is the interval
$[0,1]$ or there is a deleted digits Cantor set $\mathcal{B}$ such
that $\left(-\mathcal{F}\right)\cup\mathcal{F}=h\mathcal{B}-d_{m}/(n-1).$ \end{prop}
\begin{proof}
Suppose $\mathcal{D}=\{d_{1},d_{2},\ldots,d_{m}\},$ where $e_{j}=d_{j}/h$
is an integer for $j=1,2,\ldots,m.$ Fix $t\in\mathcal{G}.$ Suppose
$x,y$ in $\mathcal{C}$ are such that $x=y+t.$ That is such that
$t=x-y.$ Consider the $n-$ary representations $x=0.x_{1}x_{2}\cdots$
and $y=0.y_{1}y_{2}\cdots.$ Rewrite $(x-y)/h=t/h$ as \[
\sum_{k=1}^{\infty}\frac{(x_{k}-y_{k})/h}{n^{k}}=\frac{t}{h}.\]
By making suitable choices for $x$ and $y$ we can arrange that $(x_{k}-y_{k})/h$
is any sequence of points in $\mathcal{E}=\{e_{i}-e_{j}\mid1\leq i,j\leq m\}.$
So setting $z_{k}=e_{m}+(x_{k}-y_{k})/h$ we can arrange that $z_{k}$
is any sequence in $\mathcal{E}+e_{m}.$ Consequently, $z=0.z_{1}z_{2}\cdots$
can be any number in the set $\mathcal{B}$ of $n-$ary numbers with
digits in $\mathcal{E}+e_{m}.$ Thus $t/h$ can be any number number
in $\mathcal{B}-\frac{e_{m}}{n-1}.$ Hence \[
\mathcal{G}=h\mathcal{B}-\frac{d_{m}}{n-1}.\]
Note, $\{0,e_{m},2e_{m}\}\subseteq\mathcal{E}+e_{m}\subseteq\{0,1,\ldots,2e_{m}\}.$ 

If $\mathcal{E}+e_{m}=\{0,1,\ldots,n-1\},$ then $\mathcal{B}$ is
the interval $[0,1].$ Otherwise, $\mathcal{B}$ is the deleted digits
Cantor set consisting of the $n-$ary numbers with digits from $\mathcal{E}+e_{m}.$ 
\end{proof}

\section{Examples\label{sec:Concluding-Remarks}}

In this section we will present some examples illustrating that to
obtain the conclusions of Theorem \ref{thm:non-uniform} we must impose
some conditions on the digit set $\mathcal{D}.$ In Theorem \ref{thm:non-uniform}
we impose two condition on $\mathcal{D}.$ The separation condition
$d_{j+1}-d_{j}\geq2$ and the condition $d_{m}<n-1.$ In Theorem \ref{thm:uniform}
the condition $d_{m}<n-1$ is replaced by a uniformity condition. 

In the first example $\mathcal{D}$ violates both of the conditions
imposed in Theorem \ref{thm:non-uniform}. 
\begin{example}
For any $n\geq2,$ if $\mathcal{D}=\left\{ 0,1,\ldots,n-1\right\} ,$
then $\mathcal{C}$ is the closed interval $[0,1].$ Consequently,
$\mathcal{C}\cap\left(\mathcal{C}+x\right)$ either has dimension
zero or one. 
\end{example}
In the following example, $\mathcal{D}$ does not satisfy the separation
condition $d_{j+1}-d_{j}\geq2,$ but $d_{m}<n-1.$ 
\begin{example}
\label{exa:Uniform-large-dm} If $n>4$ and digit sets is $\mathcal{D}=\{0,1,\ldots,n-2\},$
then $\mathcal{F}=\left[0,n-2/n-1\right]$ by Proposition \ref{pro:F-interval}.
The Hausdorff and lower Minkowski dimensions of $\mathcal{C}\cap\left(\mathcal{C}+x\right)$
are at least $\log_{n}(n-3)$ for any $x=0.x_{1}x_{2}\cdots$ such
that $x_{1}<n-2.$ \end{example}
\begin{proof}
Since the separation condition $d_{j+1}-d_{j}\geq2$ is not satisfied
we need a variant of the analysis in Section \ref{sec:Preliminaries}.
Fix $x=0.x_{1}x_{2}\cdots.$ Let $B_{k}$ be the number of intervals
in $\mathcal{C}_{k}$ that are both in the interval case and in the
potential interval case and let $I_{k}$ be the number of intervals
in $\mathcal{C}_{k}$ that are in the interval case but not in the
potential interval case. We will ignore the intervals in $\mathcal{C}_{k}$
that are in the potential interval case but not in the interval case.

Note that $B_{0}=0$ and $I_{0}=1.$ 

\begin{figure}[h]
\includegraphics{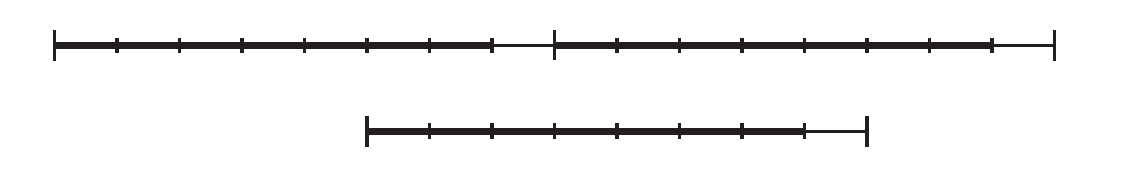}

\caption{Simultaneous interval and potential interval case, when $n=8$ and
$x_{k+1}=3$}

\label{Flo:n=00003D8-and-x=00003D3}
\end{figure}
The reader may verify that for any digit $x_{k+1}=0,1,\ldots,n-1$
we have \begin{equation}
B_{k+1}\geq(n-3)B_{k},\text{ for }k\geq0.\label{eq:example-2-0}\end{equation}
See Figure \ref{Flo:n=00003D8-and-x=00003D3} for the case $n=8$
and $x_{k+1}=3.$ The bottom line illustrates an interval in $\mathcal{C}_{k}$
and its refinement, the boldface line indicates the intervals retained
after refinement. The top line illustrates the two intervals in $\mathcal{C}_{k}+\left\lfloor x\right\rfloor _{k}.$ 

\begin{figure}[h]
\includegraphics{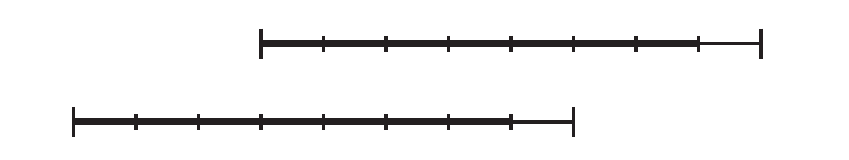}

\caption{Interval case, when $n=8$ and $x_{k+1}=3$}

\label{Flo:n=00003D8-and-x=00003D3-1}
\end{figure}
Also, as illustrated in Figure \ref{Flo:n=00003D8-and-x=00003D3-1}
\[
x_{k+1}<n-2\implies B_{k+1}\geq I_{k},\text{ for }k\geq0.\]
Consequently, if $x_{1}<n-2,$ then\begin{equation}
B_{k}\geq(n-3)^{k-1},\text{ for }k\geq1.\label{eq:example-2-1}\end{equation}
If $I$ is an interval in $\mathcal{C}_{k}$ that is both interval
case and the potential interval case, then the refinement of $I$
contains an interval of the same type in $\mathcal{C}_{k+1}.$ Hence,
by (\ref{eq:example-2-0}) and the Nested Interval Theorem, any interval
in $\mathcal{C}_{k},$ that is both in the interval case and in the
potential interval case, contains infinitely many points from $\mathcal{C}\cap\left(\mathcal{C}+x\right).$
Consequently, \begin{equation}
\mathcal{N}_{k}\left(\mathcal{C}\cap\left(\mathcal{C}+x\right)\right)\geq B_{k}\label{eq:example-2-2}\end{equation}
for all $k\geq1.$ 

Suppose $x_{1}<n-2.$ Combining (\ref{eq:example-2-1}) and (\ref{eq:example-2-2})
it follows that the lower Minkowski dimension of $\mathcal{C}\cap\left(\mathcal{C}+x\right)$
is at least $\log_{n}(n-3).$ 

Let $\mathcal{A}_{k}$ be the intervals in $\mathcal{C}_{k}$ that
are both in the interval and the potential interval case. Above we
showed the lower Minkowski dimension of $\mathcal{B}_{x}=\bigcap_{k=1}^{\infty}\bigcup_{I\in\mathcal{A}_{k}}I$
is bounded below by $\log_{n}(n-3).$ It follows from the the argument
in Sub-Section \ref{sub:Hausdorff-dimension} that the Hausdorff dimension
of $\mathcal{B}_{x}$ is also bounded below by $\log_{n}(n-3).$ Since
$\mathcal{B}_{x}\subseteq\mathcal{C}\cap\left(\mathcal{C}+x\right)$
the Hausdorff dimension of $\mathcal{C}\cap\left(\mathcal{C}+x\right)$
is also bounded below by $\log_{n}(n-3).$\end{proof}
\begin{rem}
If $n=3$ in Example \ref{exa:Uniform-large-dm} then $\mathcal{C}$
is the standard middle thirds Cantor set scaled by the factor $1/2.$
Hence the conclusions of Theorem \ref{thm:non-uniform} hold in this
case. It remains to consider the case $n=4.$ In this case an argument
similar to the one given for Example \ref{exa:Uniform-large-dm},
but also incorporating ideas from the argument given for Example \ref{exa:counter},
shows the conclusions of Theorem \ref{thm:non-uniform} fail.  
\end{rem}
The following example $\mathcal{D}$ satisfies the separation condition
$d_{j+1}-d_{j}\geq2,$ but $d_{m}=n-1.$ 
\begin{example}
\label{exa:counter} Let $n=8.$ If $\mathcal{D}=\{0,2,4,6\},$ then
we can apply Theorem \ref{thm:non-uniform}. If $\mathcal{D}=\{0,2,4,7\},$
then we cannot apply Theorem \ref{thm:non-uniform}. By Proposition
\ref{pro:F-interval} $\mathcal{F}=[0,1].$ We claim the Hausdorff
and lower Minkowski dimensions of $\mathcal{C}\cap\left(\mathcal{C}+x\right)$
are least $\log_{8}\sqrt{2}=1/6,$ for any $x=0.x_{1}x_{2}\cdots$
such that $x_{1}=3$ or $x_{1}=4.$ \end{example}
\begin{proof}
Since the separation condition $d_{j+1}-d_{j}\geq2$ is satisfied
we can use the analysis in Section \ref{sec:Preliminaries}. Fix $x=0.x_{1}x_{2}\cdots.$
Let $I_{k}$ be the number of intervals in $\mathcal{C}_{k}$ that
are in the interval case and let $P_{k}$ be the number of intervals
in $\mathcal{C}_{k}$ that are in the potential interval case.  

\begin{figure}[h]
\includegraphics{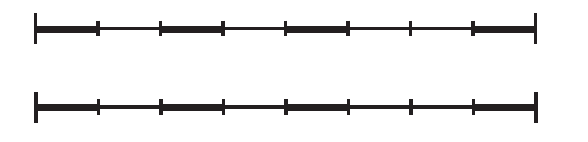}

\caption{Interval case with $x_{k+1}=0$}

\label{Flo:non-uniform-interval}
\end{figure}

We claim that for any $k\geq0$ at least one of the following four
sets of inequalities holds \begin{align*}
I_{k+1} & \geq2I_{k} & \text{and} &  & P_{k+1} & \geq P_{k}\\
I_{k+1} & \geq I_{k} & \text{and} &  & P_{k+1} & \geq2P_{k}\\
I_{k+1} & \geq2P_{k} & \text{and} &  & P_{k+1} & \geq I_{k}\\
I_{k+1} & \geq P_{k} & \text{and} &  & P_{k+1} & \geq2I_{k}\end{align*}
When $x_{k+1}=0$ this is illustrated in Figure \ref{Flo:non-uniform-interval}
and Figure \ref{Flo:non-uniform-potential}. The cases where $x_{k+1}$
is one of $1,2,5,6,7$ can be handled in the same way. 

\begin{figure}[h]
\includegraphics{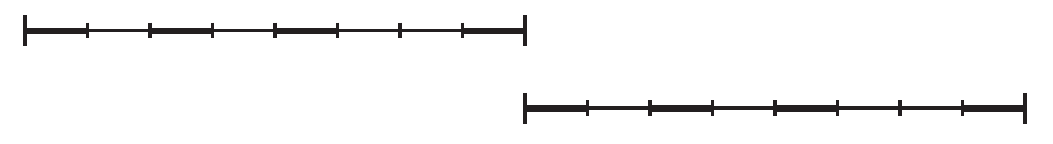}

\caption{Potential interval case with $x_{k+1}=0$}

\label{Flo:non-uniform-potential}
\end{figure}
When $x_{k+1}$ is $3$ or $4,$ illustrations similar to Figure \ref{Flo:non-uniform-interval}
and Figure \ref{Flo:non-uniform-potential} show that $I_{k+1}=I_{k}+P_{k}$
and $P_{k+1}=I_{k}+P_{k}.$ If $I_{k+1}\leq P_{k+1}$ the equalities
imply, for example, $I_{k+1}\geq2I_{k}$ and $P_{k+1}\geq P_{k}.$
A similar argument applies if $P_{k+1}<I_{k+1}.$

Assume $x_{1}=3$ or $x_{1}=4,$ then $I_{1}=P_{1}=1.$ Since at each
stage we multiply one of $I_{k}$ and $P_{k}$ by at least $2$ and
the other by at least $1,$ it follows that\[
I_{k}\geq2^{r_{k}}\text{ and }P_{k}\geq2^{k-r_{k}-1}\text{ for }k\geq1\]
where $r_{k}$ are integers satisfying $r_{1}=0$ and $r_{k}\leq r_{k+1}\leq1+r_{k}.$
Since either $r_{k}\geq(k-1)/2$ or $k-r_{k}-1\geq(k-1)/2$ we have
\[
\max\left\{ I_{k},P_{k}\right\} \geq2^{(k-1)/2}\]
when $k\geq1.$ By an argument in Example \ref{exa:Uniform-large-dm}\[
\mathcal{N}_{k}\left(\mathcal{C}\cap\left(\mathcal{C}+x\right)\right)\geq\max\left\{ I_{k},P_{k}\right\} .\]
Consequently, the lower Minkowski dimension of $\mathcal{C}\cap\left(\mathcal{C}+x\right)$
is at least $\log_{8}2^{1/2}.$ As in Example \ref{exa:Uniform-large-dm}
the argument from Sub-Section \ref{sub:Hausdorff-dimension} implies
that the Hausdorff dimension has the same lower bound.
\end{proof}

\section{Concluding Remarks\label{sec:Open-Questions}}

Suppose $n=3$ and $\mathcal{D}=\left\{ 0,2\right\} .$ In Section
\ref{sec:A-Stepping-Stone} we constructed $t=0.t_{1}t_{2}\cdots$
that do not admit a finite ternary representation such that 
$\mathcal{C}\cap\left(\mathcal{C}+t\right)$
has Hausdorff dimension $\log_{3}2.$ It follows from a result due to Nekka and Li \cite{NeLi02} than the $\log_{3}2$ dimensional Hausdorff measure of 
$\mathcal{C}\cap\left(\mathcal{C}+t\right)$
equals zero. Hence, by Hutchinson's theorem \cite{Hut81} 
$\mathcal{C}\cap\left(\mathcal{C}+t\right)$
cannot be a self-similar set satisfying the open set condition. So
for a deleted digits Cantor set $\mathcal{C}$ an interesting problem
is to characterize the $t$'s for which $\mathcal{C}\cap\left(\mathcal{C}+t\right)$
is self-similar. 

We showed that $\mathcal{G}=-\mathcal{F}\cup\mathcal{F}$ is self-similar
and if the digit set $\mathcal{D}$ is a subset of an arithmetic progression,
then $\mathcal{G}$ is similar to a deleted digits Cantor set. Does
there exists $n$ and $\mathcal{D}$ such that $\mathcal{F}$ is not
an interval, yet $\mathcal{F}$ is self-similar or even a deleted
digits Cantor set $\mathcal{F}=\mathcal{C}_{m,\mathcal{E}}?$ 

The examples in Section \ref{sec:Concluding-Remarks} suggests that,
when the assumptions in Theorem's \ref{thm:non-uniform} and \ref{thm:uniform}
fail, it may still be possible to find $a,$ perhaps in terms of $0.t_{1}t_{2}\cdots t_{n},$
such that for any $a\leq\alpha\leq1,$ there is a {}``completion''
$t=0.t_{1}t_{2}\cdots t_{n}t_{n+1}\cdots$ for which $\mathcal{C}\cap\left(\mathcal{C}+t\right)$
has dimention $\alpha\log_{n}m.$

\end{document}